\documentclass[submission]{FPSAC2018-arxiv}
\usepackage{pgf,tikz,mathrsfs}
\usetikzlibrary{arrows}

\newcommand{\bk}{\ensuremath{\Bbbk}}
\newcommand{\scrm}{\ensuremath{\mathscr{M}}}
\newcommand{\m}{\ensuremath{\mathfrak{m}}}
\newcommand{\B}{\ensuremath{\mathcal{B}}}

\newcommand{\tangle}[1]{\ensuremath{\langle #1 \rangle}}
\newcommand{\commen}[1]{}

\DeclareMathOperator{\Prim}{\ensuremath{Prim}}

\newtheorem{theorem}[equation]{Theorem}
\newtheorem{cor}[equation]{Corollary}
\newtheorem{lemma}[equation]{Lemma}
\newtheorem{proposition}[equation]{Proposition}
\newtheorem{defn}[equation]{Definition}
\newtheorem{remark}[equation]{Remark}

\numberwithin{equation}{section}

\newcommand{\Z}{\mathbb{Z}}

\title{Generalized nil-Coxeter algebras}

\author{Apoorva
Khare\thanks{\href{mailto:khare@iisc.ac.in}{khare@iisc.ac.in}. A.K.~was
partially supported by an Infosys Young Investigator Award.}\addressmark{1}}

\address{\addressmark{1}Indian Institute of Science;
Analysis and Probability Research Group; Bangalore, India}

\received{\today}

\abstract{Motivated by work of Coxeter (1957), we study a class of
algebras associated to Coxeter groups, which we term `generalized
nil-Coxeter algebras'. We construct the first finite-dimensional examples
other than usual nil-Coxeter algebras; these form a $2$-parameter type
$A$ family that we term $NC_A(n,d)$. We explore the combinatorial
properties of these algebras, including the Coxeter word basis, length
function, maximal words, and their connection to Khovanov's
categorification of the Weyl algebra.

Our broader motivation arises from complex reflection groups and the
Brou\'e--Malle--Rouquier freeness conjecture (1998). With generic Hecke
algebras over real and complex groups in mind, we show that the `first'
finite-dimensional examples $NC_A(n,d)$ are in fact the only ones,
outside of the usual nil-Coxeter algebras.
The proofs use a diagrammatic calculus akin to crystal theory.
}


\begin{document}

\maketitle

\section{Introduction and main results}

We study a new class of finite-dimensional algebras arising out of
Coxeter theory, with connections to old work by Coxeter and new work on
generic Hecke algebras, combinatorics, and categorification. We work
throughout over a ground field $\bk$ for ease of exposition, although our
results hold over any commutative unital ground ring.

We begin with background and notation.
Real reflection groups $W$ and their Iwahori--Hecke algebras
$\mathcal{H}_W(q)$ are classical objects that have long been studied in
algebraic combinatorics, representation theory, and mathematical physics.
Recall that every Coxeter group is specified by a Coxeter matrix $M \in
\Z^{I \times I}$, with finite index set $I$ and entries $m_{ii} = 2
\leqslant m_{ij} \leqslant \infty\ \forall i \neq j$. The corresponding
Artin monoid $\B_M^{\geqslant 0}$ has generators $\{ T_i : i \in I \}$,
and braid relations $T_i T_j T_i \cdots = T_j T_i T_j \cdots$ with
$m_{ij}$ factors on each side whenever $m_{ij} < \infty$. We will denote
the corresponding Coxeter group by $W(M)$.

Three prominent algebras associated to $W(M)$ are its group algebra $\bk
W(M)$, its $0$-Hecke algebra, and its nil-Coxeter algebra $NC(M)$ (also
known in the literature as the nil Hecke ring, nil Coxeter algebra, and
nilCoxeter algebra). All three algebras are quotients of the monoid
algebra $\bk \B_M^{\geqslant 0}$ by quadratic relations for the $T_i$,
and are `generic Hecke algebras' \cite[Chapter 7]{Hum}. Among such
algebras, the $T_i$ satisfy homogeneous relations only in the case of the
nil-Coxeter algebra $NC(M)$. Using this, one shows that
\[
NC(M) := \bk \B_M^{\geqslant 0} / (T_i^2 : i \in I)
\]
is the monoid algebra of a monoid with $|W(M)|+1$ elements, say $\{ T_w :
w \in W(M) \} \sqcup \{ O_{W(M)} \}$ quotiented by the `absorbing'
central ideal $\bk O_{W(M)}$.
Nil-Coxeter algebras were introduced by Fomin and Stanley \cite{FS}, and
are related to flag varieties \cite{KK}, symmetric function theory
\cite{BSS}, and categorification \cite{Kho}.

We now present a larger family of algebras, which constitute the main
object of study. Note that the algebras $NC(M)$ are the associated graded
versions of the group algebra $\bk W(M)$; indeed, taking the
top-degree component of the non-homogeneous relations $T_i^2 = 1$ yields
the nil-Coxeter relations $T_i^2 = 0$. The following construction is
motivated by both real and complex reflection groups, and allows the
nilpotence degree to vary.

\begin{defn}\label{Dnilcox}
Define a {\em generalized Coxeter matrix} to be a symmetric
matrix $M := (m_{ij})_{i,j \in I}$ with $I$ finite, $m_{ii} < \infty\
\forall i \in I$, and $2 \leqslant m_{ij} \leqslant \infty\ \forall i
\neq j$. Now fix such a matrix $M$.
\begin{enumerate}
\item Given an integer tuple ${\bf d} = (d_i)_{i \in I}$ with all $d_i
\geqslant 2$, let $M({\bf d})$ denote the matrix where the diagonal in
$M$ is replaced by the coordinates of ${\bf d}$. Let $M_2 :=
M((2,\dots,2))$.

\item The {\em generalized Coxeter group} $W(M)$ is the group
generated by $\{ s_i : i \in I \}$ modulo the braid relations $s_i s_j
s_i \cdots = s_j s_i s_j \cdots$ whenever $m_{ij} < \infty$, and the
relations $s_i^{m_{ii}} = 1\ \forall i$.

\item Define the corresponding {\em generalized nil-Coxeter algebra}
to be:
\begin{equation}\label{Egen}
NC(M) := \frac{\bk \tangle{T_i, i \in I}}
{(\underbrace{T_i T_j T_i \cdots}_{m_{ij}\ times} =
\underbrace{T_j T_i T_j \cdots}_{m_{ij}\ times}, \
T_i^{m_{ii}} = 0, \ \forall i \neq j \in I)}
= \frac{\bk \B_{M_2}^{\geqslant 0}}{(T_i^{m_{ii}} = 0\ \forall i)},
\end{equation}
where we omit the braid relation $T_i T_j T_i \cdots = T_j T_i T_j
\cdots$ if $m_{ij} = \infty$.

\item As an important special case, we denote by $M_{A_n}$ the usual type
$A$ Coxeter matrix with $|I| = n$, given by: $m_{ij} = 3$ if $|i-j|=1$,
and $2$ otherwise.
\end{enumerate}
\end{defn}

Working with generalized nil-Coxeter algebras $NC(M)$ yields a larger
class of objects than the corresponding groups $W(M)$. For example, Marin
\cite{Ma} has shown that in rank $2$ in type $A$, the algebra
$NC(M_{A_2}((3,n)))$ is not finite-dimensional for $n \geqslant 3$, in
particular for even $n$. However, the corresponding generalized Coxeter
group $W(M_{A_2}((3,n)))$ is trivial for $3 \nmid n$, since in it the
generators $s_1, s_2$ are conjugate, hence have equal orders.

In fact, this reasoning shows that for for all integers $d_1, \dots, d_n
\geqslant 2$, we have
\[
W(M_{A_n}({\bf d})) = W(M_{A_n}((d,\dots,d))), \quad \text{where} \quad d
= \gcd(d_1, \dots, d_n).
\]
Now it is natural to ask for which integers $n,d \geqslant 2$ is the
group $W(M_{A_n}((d,\dots,d)))$ finite -- and what is its order. These
questions were considered by Coxeter \cite{Co}, and he proved that $W =
W(M_{A_n}((d, \dots, d)))$ is finite if and only if $\frac{1}{n} +
\frac{1}{d} > \frac{1}{2}$; moreover, in this case $W$ has size $\left(
\frac{1}{n} + \frac{1}{d} - \frac{1}{2} \right)^{1-n} \cdot n! /
n^{n-1}$. In his thesis \cite{Ko}, Koster extended Coxeter's results to
classify all finite generalized Coxeter groups; apart from the finite
`usual' Coxeter groups, one obtains precisely the Shephard groups.

In a parallel vein to these works, we explore for which matrices is the
algebra $NC(M)$ finite-dimensional. In this we are also strongly
motivated by the larger picture, which involves \textit{complex}
reflection groups and the \textit{BMR freeness conjecture} \cite{BMR2}.
We elaborate on these motivations presently; for now we remark that since
complex reflections can have order $\geqslant 3$, working with them
provides a natural reason to define and study generalized nil-Coxeter
algebras.

Returning to real groups: recall that `usual' nil-Coxeter algebras
$NC(M((2,\dots,2)))$ are finite-dimensional precisely for finite Coxeter
groups, since for `usual' Coxeter matrices $M_2 := M((2,\dots,2))$ one
has $\dim NC(M_2) = |W(M_2)|$. To our knowledge there are no other
finite-dimensional examples $NC(M)$ known to date.

Our first main result parallels Coxeter's construction, and exhibits the
first such `non-usual' family of finite-dimensional algebras $NC(M)$ in
type $A$:

\begin{theorem}\label{ThmA}
For integers $n \geqslant 1$ and $d \geqslant 2$, define the
$\bk$-algebra
\begin{equation}
NC_A(n,d) := NC(M_{A_n}((2,\dots,2,d))).
\end{equation}
Thus, $NC_A(n,d)$ has generators $T_1, \dots, T_n$, with relations:
\begin{alignat}{5}
T_i T_{i+1} T_i & = T_{i+1} T_i T_{i+1}, & & \forall\ 0 < i < n;\\
T_i T_j & = T_j T_i, & & \forall\ |i-j| > 1;\\
T_1^2 & = \cdots = T_{n-1}^2 & = T_n^d = &\ 0.
\end{alignat}

\noindent Then $NC_A(n,d)$ has a Coxeter word basis of
$n! (1 + n(d-1))$ generators
\[
\{ T_w : \ w \in S_n \} \ \sqcup \
\{ T_w T_n^k T_{n-1} T_{n-2} \cdots T_{m+1} T_m :\ w \in S_n,\ k \in
[1,d-1],\ m \in [1,n] \}.
\]
In particular, the subalgebra $R_l$ generated by
$T_1, \dots, T_l$ is isomorphic to the type $A$ nil-Coxeter algebra
$NC(M_{A_l}((2,\dots,2)))$, for all $0 < l < n$.
\end{theorem}

\begin{remark}
We adopt the following notation in the sequel without further reference:
let
\begin{equation}
w_\circ \in S_{n+1}, \quad w'_\circ \in S_n
\quad \text{denote the respective longest elements},
\end{equation}

\noindent where the symmetric group $S_{l+1}$ corresponds to the basis of
the algebra $R_l$ for $l = n-1, n$.
\end{remark}

In a later section, we will discuss additional properties of the algebras
$NC_A(n,d)$, including identifying the `maximal' words, and exploring the
Frobenius property.

\subsection*{Classification of finite-dimensional nil-Coxeter algebras}

Our next main result classifies the matrices $M$ for which the
generalized nil-Coxeter algebra $NC(M)$ is finite-dimensional. In
combinatorics and in algebra, classifying Coxeter-type objects of finite
size, dimension, or type is a problem of significant classical as well as
modern interest. Such settings include real and complex reflection groups
\cite{Cox,Cox2,ST} and associated Hecke algebras; finite type quivers,
simple Lie algebras, the McKay--Slodowy correspondence, and Kleinian
singularities (as well as the above results by Coxeter and Koster). The
recent classification of finite-dimensional pointed Hopf algebras
\cite{AnSc} reveals connections to small quantum groups. Even more
recently, the classification of finite-dimensional Nichols algebras has
been well-received (see \cite{HV2} and the references therein); some
ingredients used in proving those results show up in the present work as
well.

We now classify the generalized Coxeter matrices $M$ for which $NC(M)$ is
finite-dimensional. Remarkably, outside of the usual nil-Coxeter
algebras, our first family of examples $NC_A(n,d)$ turns out to be the
only one:

\begin{theorem}\label{ThmC}
Suppose $W$ is a Coxeter group with connected Dynkin diagram.
Fix an integer vector ${\bf d}$ with $d_i \geqslant 2\ \forall i$, i.e.,
a generalized Coxeter matrix $M({\bf d})$.
The following are equivalent:
\begin{enumerate}
\item The generalized nil-Coxeter algebra $NC(M({\bf d}))$ is
finite-dimensional.

\item Either $W$ is a finite Coxeter group and $d_i = 2 \ \forall i$, or
$W$ is of type $A_n$ and ${\bf d} = (2, \dots, 2, d)$ or $(d, 2, \dots,
2)$ for some $d>2$.
\end{enumerate}
\end{theorem}

\begin{remark}
The above results are characteristic-free; in fact they hold over
arbitrary ground rings $\bk$, in which case Theorem \ref{ThmA} yields a
$\bk$-basis of the free $\bk$-module $NC_A(n,d)$; and Theorem \ref{ThmC}
classifies the finitely generated $\bk$-algebras $NC(M)$. In sketching
the proofs of these results below, we will continue to assume $\bk$ is a
field; for the general case over a ring $\bk$, for full details, and for
further ramifications, we refer the reader to \cite{Kh}, of which this
note is an extended abstract.
\end{remark}

Before proceeding further, we mention another strong motivation for
Theorem \ref{ThmC}, arising from generic Hecke algebras over
\textit{complex} reflection groups.
As mentioned above, the varying nilpotence degree of the $T_i$ is natural
in the setting of complex reflection groups $W$.
A prominent area of research has been the study of the associated generic
Hecke algebras $\mathcal{H}_W$ and the Brou\'e--Malle--Rouquier freeness
conjecture \cite{BMR2}.
The conjecture says that $\mathcal{H}_W$ is a free $R$-module of rank
$|W|$, where $R$ is the ground ring. See also its recent resolution in
characteristic zero \cite{Et}, and the references therein.

In studying these topics, Marin \cite{Ma} remarks that the lack of
nil-Coxeter algebras of dimension $|W|$ is a striking difference between
complex and real reflection groups $W$. This was verified in some cases
in \textit{loc.~cit.}; and it motivated us to define generalized
nil-Coxeter algebras over all complex reflection groups. We do so in
\cite{Kh}, and then completely classify the finite-dimensional algebras
over all such groups. Remarkably, Theorem \ref{ThmC} extends to all
complex $W$ as well, and the only finite-dimensional families are real
(usual) nil-Coxeter algebras, and the family $NC_A(n,d)$. In particular,
this shows the above statement of Marin.

\begin{remark}
Our result holds even more generally: following the classification of
finite complex reflection groups in the celebrated work \cite{ST}, Popov
classified in \cite{Po1} the \textit{infinite} discrete groups generated
by unitary reflections. In \cite{Kh} we extend Theorem \ref{ThmC} to also
cover all of these groups; once again, we show there are no
finite-dimensional nil-Coxeter analogues.
\end{remark}

The equidimensionality (or not) of $\mathcal{H}_W$ and its nil-Coxeter
analogue amounts to whether the former -- a filtered algebra -- is a
\textit{flat} deformation of the latter, which is $\Z^{\geqslant
0}$-graded. The study of flat deformations goes back to classical work of
Gerstenhaber, and also by Braverman--Gaitsgory, Drinfeld,
Etingof--Ginzburg, and the recent program by Shepler and Witherspoon; see
\cite{SW2,Kh} for more on this. In this formalism, Theorem \ref{ThmC} --
or its extension to complex groups -- says that over complex reflection
groups, generic Hecke algebras are not flat deformations of their
nil-Coxeter analogues. This is in stark contrast to the real case, where
$\dim NC(M) = |W(M)|$.

\section{A finite-dimensional generalized nil-Coxeter algebra}

We now outline the proof of Theorem \ref{ThmA}, using a diagrammatic
calculus as well as braid monoid computations. Note that $NC_A(1,d) =
\bk[T_1] / (T_1^d)$, while $NC_A(n,2)$ is the usual type $A$ nil-Coxeter
algebra, for which the theorem is well-known (see e.g. \cite{Hum}). Thus,
in this section we will assume $d \geqslant 3$ and $n \geqslant 2$.

We begin by showing that the set from Theorem \ref{ThmA} spans
$NC_A(n,d)$. As a first step:

\begin{lemma}\label{L1}
A word in the generators $T_i$ either vanishes in $NC_A(n,d)$, or can be
equated with a word in which all occurrences of $T_n$ are successive.
\end{lemma}

\begin{proof}
Suppose a word $\mathcal{T}$ has a sub-word of the form $T_n^a T_{i_1}
\cdots T_{i_k} T_n^b$ for some $a,b > 0$, with $0 < i_j < n\ \forall j$.
Using the relations, we may assume the above representation of
$\mathcal{T}$ is such that $k$ is minimal. Thus $i_1 = i_k = n-1$, $i_2 =
i_{k-1} = n-2$, and so on (else we may push some $T_{i_j}$ outside of the
sub-string). Hence the sub-string is of the form
\[
T_{n-1} T_{n-2} \cdots T_{m+1} T_m T_{m+1} \cdots T_{n-2} T_{n-1},
\quad \text{for some } 1 \leqslant m \leqslant n-1.
\]
Now one shows by descending induction on $m \leqslant n-1$ that in the
Artin monoid $\B_{M_{A_n}}^{\geqslant 0}$,
\[
T_{n-1} \cdots T_m \cdots T_{n-1} = T_m T_{m+1} \cdots T_{n-2} T_{n-1}
T_{n-2} \cdots T_{m+1} T_m.
\]
Hence,
\begin{align}\label{Ereduce}
T_n^a \cdot (T_{n-1} \cdots T_m \cdots T_{n-1}) \cdot T_n^b
= &\ T_n^a \cdot (T_m \cdots T_{n-2} T_{n-1} T_{n-2} \cdots T_m) \cdot
T_n^b\\
= &\ (T_m \cdots T_{n-2}) T_n^{a-1} (T_n T_{n-1} T_n) T_n^{b-1} (T_{n-2}
\cdots T_m)\notag\\
= &\ (T_m \cdots T_{n-2}) T_n^{a-1} (T_{n-1} T_n T_{n-1}) T_n^{b-1}
(T_{n-2} \cdots T_m).\notag
\end{align}

\noindent If $a,b \leqslant 1$ then the lemma follows. If instead $b>1$
then this expression contains as a substring $T_{n-1} (T_n T_{n-1} T_n)
= T_{n-1}^2 T_n T_{n-1} = 0$, so we are done. Similarly if $a>1$.
\end{proof}

Now the subalgebra $R_{n-1}$ generated by $T_1, \dots, T_{n-1}$ satisfies
the relations of the usual nil-Coxeter algebra $NC_A(n-1,2)$, so the
words $\{ T_w : w \in S_n \}$ span it. By \eqref{Ereduce}, every nonzero
word not in $R_{n-1}$ is of the form $T_w T_n^k T_{w'}$; writing
$T_{w'}$ as a sub-string of minimal length, by above we
may rewrite the word such that $T_{w'} = T_{n-1} \cdots T_m$. Hence,
\[
NC_A(n,d) = R_{n-1} + \sum_{k=1}^{d-1} \sum_{m=1}^n R_{n-1} \cdot T_n^k
\cdot (T_{n-1} \cdots T_m).
\]
Since $\dim R_{n-1} \leqslant n!$, the upper bound on $\dim NC_A(n,d)$
follows.

The proof of the converse -- i.e., linear independence of the claimed
word basis -- repeatedly uses some results about the permutation group
$S_n$ and its nil-Coxeter algebra:

\begin{lemma}\label{Lsymm}
Suppose $W = S_n$, with simple reflections $s_1, \dots, s_{n-1}$ labelled
as usual. Let $S_{n-1}$ be generated by $s_1, \dots, s_{n-2}$; then for
all $w \in S_n \setminus S_{n-1}$, $w$ has a reduced expression as $w =
w' s_{n-1} \cdots s_{m'}$, where $w' \in S_{n-1}$ and $m' \in [1,n-1]$
are unique. Given such an element $w \in S_n$, we have in the usual
nil-Coxeter algebra $NC(M_{A_n}((2,\dots,2)))$:
\begin{equation}\label{EnilcoxA}
T_n \cdot T_w \cdot T_n \cdots T_m = \begin{cases}
T_{w'} T_{n-1} \cdots T_{m-1} \cdot T_n \cdots T_{m'}, & \qquad \text{if
} m' < m,\\
0 & \qquad \text{otherwise}.
\end{cases}
\end{equation}
\end{lemma}

Now we introduce a diagrammatic calculus reminiscent of crystal theory
from combinatorics and quantum groups. For simplicity, we begin by
presenting the $n=2$ case. Let $\scrm$ be a $\bk$-vector space, with
basis given by the nodes in Figure \ref{Fig1}.

\begin{figure}[ht]
\hspace*{7mm}\begin{tikzpicture}[line cap=round,line join=round,>=triangle 45,x=1.0cm,y=1.0cm]
\draw(1,2) circle (0.5cm);
\draw(1,4) circle (0.5cm);
\draw(5,2) circle (0.5cm);
\draw(5,4) circle (0.5cm);
\draw(7,2) circle (0.5cm);
\draw(7,4) circle (0.5cm);
\draw(9,2) circle (0.5cm);
\draw(9,4) circle (0.5cm);
\draw(11,2) circle (0.5cm);
\draw(11,4) circle (0.5cm);
\draw(15,2) circle (0.5cm);
\draw(15,4) circle (0.5cm);
\draw(8,5.5) circle (0.5cm);
\draw(8,0.5) circle (0.5cm);
\draw [->] (1,2.5) -- (1,3.5);
\draw [->] (2.5,2) -- (1.5,2);
\draw [->] (4.5,2) -- (3.5,2);
\draw [->] (5,2.5) -- (5,3.5);
\draw [->] (6.5,2) -- (5.5,2);
\draw [->] (7,3.5) -- (7,2.5);
\draw [->] (9,3.5) -- (9,2.5);
\draw [->] (9.5,4) -- (10.5,4);
\draw [->] (13.5,4) -- (14.5,4);
\draw [->] (11,3.5) -- (11,2.5);
\draw [->] (11.5,4) -- (12.5,4);
\draw [->] (15,3.5) -- (15,2.5);
\draw [->] (7.5,5.2) -- (7.1,4.55);
\draw [->] (8.5,5.2) -- (8.9,4.55);
\draw [->] (7.1,1.45) -- (7.5,0.8);
\draw [->] (8.9,1.45) -- (8.5,0.8);
\draw (0.52,2.4) node[anchor=north west] {$ 2^{d'} 1 $};
\draw (0.46,4.4) node[anchor=north west] {$ 1 2^{d'} 1 $};
\draw (4.58,2.35) node[anchor=north west] {$ 2^2 1 $};
\draw (4.46,4.35) node[anchor=north west] {$ 1 2^2 1 $};
\draw (6.65,2.27) node[anchor=north west] {$ 21 $};
\draw (6.75,4.3) node[anchor=north west] {$ 1 $};
\draw (8.65,2.27) node[anchor=north west] {$ 12 $};
\draw (8.75,4.3) node[anchor=north west] {$ 2 $};
\draw (10.58,2.35) node[anchor=north west] {$ 1 2^2 $};
\draw (10.7,4.35) node[anchor=north west] {$ 2^2 $};
\draw (14.5,2.35) node[anchor=north west] {$ 1 2^{d'} $};
\draw (14.7,4.35) node[anchor=north west] {$ 2^{d'} $};
\draw (2.7,2.22) node[anchor=north west] {$ \cdots $};
\draw (2.7,4.22) node[anchor=north west] {$ \cdots $};
\draw (12.7,2.22) node[anchor=north west] {$ \cdots $};
\draw (12.7,4.22) node[anchor=north west] {$ \cdots $};
\draw (7.76,5.8) node[anchor=north west] {$ \emptyset $};
\draw (7.65,0.7) node[anchor=north west] {$ w_\circ $};
\draw (0.6,3.2) node[anchor=north west] {$1$};
\draw (1.8,2) node[anchor=north west] {$2$};
\draw (3.8,2) node[anchor=north west] {$2$};
\draw (4.6,3.2) node[anchor=north west] {$1$};
\draw (5.8,2) node[anchor=north west] {$2$};
\draw (6.6,3.4) node[anchor=north west] {$2$};
\draw (6.8,1.4) node[anchor=north west] {$1$};
\draw (6.9,5.35) node[anchor=north west] {$1$};
\draw (8.75,1.4) node[anchor=north west] {$2$};
\draw (8.7,5.35) node[anchor=north west] {$2$};
\draw (8.95,3.4) node[anchor=north west] {$1$};
\draw (9.7,4.5) node[anchor=north west] {$2$};
\draw (10.95,3.4) node[anchor=north west] {$1$};
\draw (11.7,4.5) node[anchor=north west] {$2$};
\draw (13.7,4.5) node[anchor=north west] {$2$};
\draw (14.95,3.4) node[anchor=north west] {$1$};
\end{tikzpicture}
\caption{Regular representation for $NC_A(2,d)$, with $d' = d-1$}
\label{Fig1}
\end{figure}
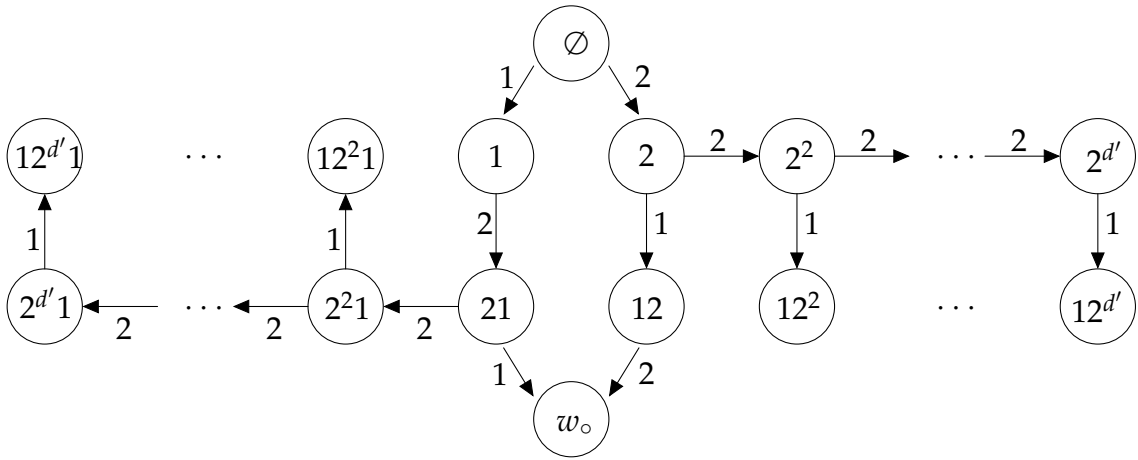

In this figure, the node $2^2 1$ should be thought of as applying $T_2^2
T_1$ to the generating basis vector/node $\emptyset$; similarly for all
other nodes. The arrows show the action of $T_1, T_2$ on the basis
vectors (i.e., nodes), and the lack of an arrow labeled $i$ with source
$v \in \scrm$ means $T_i v = 0$. Now verify by inspection that the
relations in $NC_A(2,d)$ are satisfied in ${\rm End}_\bk(\scrm)$, whence
$\scrm$ is a cyclic $NC_A(2,d)$-module generated by the vector
$\emptyset$ -- in fact, the regular representation. This gives the
desired result for $NC_A(2,d)$.

For general $n \geqslant 2$, the strategy is similar but with more
involved notation. For $w \in S_n$, let $T_w$ denote the (well-defined)
word in $T_1, \dots, T_{n-1} \in NC_A(n,d)$. Now define a vector space
$\scrm$ with basis given by \eqref{Ebasis} and $NC_A(n,d)$-action as in
Figure \ref{Fig2} below:
\begin{equation}\label{Ebasis}
\B := \{ B(w,k,m) : w \in S_n,\ k \in [1, d-1],\ m \in [1,n] \} \sqcup \{
B(w) : w \in S_n \}.
\end{equation}

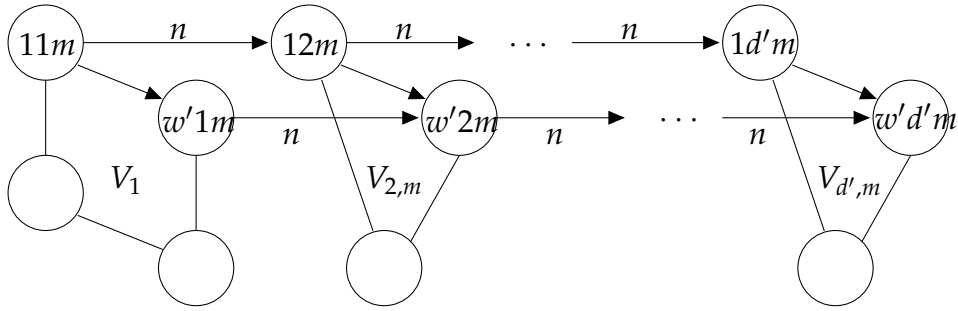
\begin{figure}[ht]
\hspace*{2cm}\begin{tikzpicture}[line cap=round,line join=round,>=triangle 45,x=1.0cm,y=1.0cm]
\draw(1.5,1.5) circle (0.5cm);
\draw(1.5,3.5) circle (0.5cm);
\draw(3.5,2.5) circle (0.5cm);
\draw(3.5,0.5) circle (0.5cm);
\draw(5,3.5) circle (0.5cm);
\draw(6,0.5) circle (0.5cm);
\draw(7,2.5) circle (0.5cm);
\draw(11,3.5) circle (0.5cm);
\draw(12,0.5) circle (0.5cm);
\draw(13,2.5) circle (0.5cm);
\draw [->] (1.93,3.2) -- (3.05,2.75);
\draw [->] (5.43,3.2) -- (6.55,2.75);
\draw [->] (11.43,3.2) -- (12.55,2.75);
\draw [->] (2,3.5) -- (4.46,3.5);
\draw [->] (4,2.5) -- (6.46,2.5);
\draw [->] (5.5,3.5) -- (7.2,3.5);
\draw [->] (7.5,2.5) -- (9.2,2.5);
\draw [->] (8.5,3.5) -- (10.46,3.5);
\draw [->] (10.5,2.5) -- (12.46,2.5);
\draw (1.5,3)-- (1.5,2);
\draw (1.93,1.2)-- (3.05,0.75);
\draw (3.5,2)-- (3.5,1);
\draw (5.17,3)-- (5.85,1);
\draw (7,2)-- (6.35,0.85);
\draw (11.17,3)-- (11.85,1);
\draw (13,2)-- (12.35,0.85);
\draw (1,3.78) node[anchor=north west] {$ 11m $};
\draw (2.9,2.8) node[anchor=north west] {$ w'1m $};
\draw (4.5,3.78) node[anchor=north west] {$ 12m $};
\draw (6.4,2.8) node[anchor=north west] {$ w'2m $};
\draw (10.45,3.83) node[anchor=north west] {$ 1d'm $};
\draw (12.37,2.84) node[anchor=north west] {$ w'd'm $};
\draw (7.5,3.71) node[anchor=north west] {$ \cdots $};
\draw (9.5,2.71) node[anchor=north west] {$ \cdots $};
\draw (2.2,2) node[anchor=north west] {$ V_1 $};
\draw (5.6,2) node[anchor=north west] {$ V_{2,m} $};
\draw (11.6,2) node[anchor=north west] {$ V_{d',m} $};
\draw (4.5,2.5) node[anchor=north west] {$n$};
\draw (8,2.5) node[anchor=north west] {$n$};
\draw (10.7,2.5) node[anchor=north west] {$n$};
\draw (3,3.9) node[anchor=north west] {$n$};
\draw (6,3.9) node[anchor=north west] {$n$};
\draw (9,3.9) node[anchor=north west] {$n$};
\end{tikzpicture}
\caption{Regular representation for $NC_A(n,d)$, with $d' = d-1$}
\label{Fig2}
\end{figure}

Note that $\dim_\bk \scrm = n! (1 + n(d-1))$; that the basis vectors in
\eqref{Ebasis} are to be thought of as akin to $T_w T_n^k T_{n-1} \cdots
T_m$ and $T_w$ respectively; and the nodes $(wkm), (w)$ precisely denote
the basis vectors $B(w,k,m), B(w)$ respectively.

Now let $V_1$ denote the span of the vectors $\{ B(w) : w \in S_{n-1} \}
\sqcup \{ B(w,1,m) : w \in S_{n-1}, m \in [1,n] \}$. These vectors are in
bijection with the word basis of the usual nil-Coxeter algebra
$NC_A(n,2)$. Similarly for $k \in [1,d-1]$ and $m \in [1,n]$, define
$V_{k,m}$ to be the span of the vectors $B(w,k,m), w \in S_n$.

Now we define the $NC_A(n,d)$-action on $\scrm$. First for the action of
$T_1, \dots, T_{n-1}$, write $V_{1,n+1} := {\rm span} \{ B(w) : w \in S_n
\}$; and equip each space $V(k,m)$ for $k \in [1,d-1], m \in [1,n]$ and
also $V(1,n+1)$ with the structure of the regular representation of
$R_{n-1}$. Next, if $w \in S_{n-1}$, then we define
\[
T_n \cdot B(w,k,m) := {\bf 1}(k \leqslant d-2) B(w,k+1,m), \qquad
T_n \cdot B(w) := B(w,1,n).
\]
Now suppose $w \in S_n \setminus S_{n-1}$.
Using Lemma \ref{Lsymm}, write $w = w' s_{n-1} \cdots s_m$; then $m
\leqslant n-1$. Define $T_n \cdot B(w,k,m) := 0$ if $k>1$; also set
$T_n \cdot B(w) := B(w',1,m')$; finally,
\begin{equation}\label{Ereln}
T_n \cdot B(w,1,m) := \begin{cases}
B(w' s_{n-1} \cdots s_{m-1},1,m'), & \qquad \text{if } m' < m,\\ 0 &
\qquad \text{otherwise}.
\end{cases}
\end{equation}

Then some involved computations using the identities mentioned above show
that the proposed action indeed equips $\scrm$ with the structure of a
cyclic $NC_A(n,d)$-module, generated by $B(1)$. This allows us to
complete the proof of Theorem \ref{ThmA}. \qed

\section{Further properties}

We now discuss several additional properties of the algebras $NC_A(n,d)$.
For proofs of results in this section, we refer the reader to \cite{Kh}.
The first set of properties shows how these algebras resemble usual
nil-Coxeter algebras. 

\begin{theorem}[see \cite{Kh}]\label{ThmB}
Fix integers $n \geqslant 1$ and $d \geqslant 2$.
\begin{enumerate}
\item The algebra $NC_A(n,d)$ has a length function that restricts to the
usual length function $\ell_{A_{n-1}}$ on $R_{n-1} \simeq
NC_{A_{n-1}}((2,\dots,2))$ (from Theorem \ref{ThmA}), and
\begin{equation}\label{Elength}
\ell(T_w T_n^k T_{n-1} \cdots T_m) = \ell_{A_{n-1}}(w) + k + n-m,
\end{equation}
for all $w \in S_n$, $k \in [1,d-1]$, and $m \in [1,n]$.

\item There is a unique longest word $T_{w'_\circ} T_n^{d-1} T_{n-1}
\cdots T_1$ of length
\[
l_{n,d} := \ell_{A_{n-1}}(w'_\circ) + d+n-2.
\]

\item The algebra $NC_A(n,d)$ is local, with unique maximal
(augmentation) ideal $\m$ generated by $T_1, \dots, T_n$. The ideal $\m$
is nilpotent with $\m^{1 + l_{n,d}} = 0$.
\end{enumerate}
\end{theorem}

Thus there is a variant of the Coxeter word length, as well as a unique
longest word and nilpotent augmentation ideal. As an immediate
consequence, one can compute the Hilbert polynomial of the graded algebra
$NC_A(n,d)$:

\begin{cor}\label{Chilb}
If $T_1, \dots, T_n$ all have degree $1$, then $NC_A(n,d)$ has
Hilbert--Poincar\'e series
\[
[n]_q! \; (1 + [n]_q \; [d-1]_q), \qquad \text{where } [n]_q :=
\frac{q^n-1}{q-1}, \ [n]_q! := \prod_{j=1}^n [j]_q.
\]
\end{cor}

Here, we also use the standard result that the usual nil-Coxeter algebra
$NC_A(n,2)$ has Hilbert--Poincar\'e series $[n]_q!$ (see e.g.~\cite[\S
3.12, 3.15]{Hum}).

Having discussed similarities with usual nil-Coxeter algebras, we next
present certain differences in structure. For any generalized Coxeter
matrix $M$, define $x \in NC(M)$ to be \textit{left-primitive} if $T_i x
= 0\ \forall i \in I$. Similarly define \textit{right-primitive}
elements; and an element that is both is said to be \textit{primitive}.
We denote these subspaces of $NC(M)$ by
\[
\Prim_L(NC(M)), \quad \Prim_R(NC(M)), \quad \Prim(NC(M)).
\]

\begin{proposition}[see \cite{Kh}]\label{Pprim}
Every generalized nil-Coxeter algebra $NC(M)$ is equipped with an
anti-involution $\theta$ that fixes each generator $T_i$. Now $\theta$ is
an isomorphism $: \Prim_L(NC(M)) \longleftrightarrow \Prim_R(NC(M))$.
Moreover, the following hold.
\begin{enumerate}
\item If $NC(M) = NC_A(1,d)$, then
\[
\Prim_L(NC(M)) = \Prim_R(NC(M)) = \Prim(NC(M)) = \bk \cdot T_1^{d-1}.
\]

\item If $NC(M) = NC_A(n,d)$ with $n \geqslant 2$ and $d \geqslant 2$,
then:
\begin{enumerate}
\item $\Prim_L(NC(M))$ is spanned by $T_{w_\circ} := T_{w'_\circ} T_n
T_{n-1} \cdots T_1$ and the $n(d-2)$ words
\[
\{ T_{w'_\circ} T_n^k T_{n-1} \cdots T_m : \ k \in [2, d-1], \ m \in
[1,n] \}.
\]

\item $\Prim(NC(M))$ is spanned by the words $T_{w'_\circ} T_n^k
T_{n-1} \cdots T_1$, where $1 \leqslant k \leqslant d-1$.
\end{enumerate}
\end{enumerate}

\noindent In all cases, the map $\theta$ fixes both $\Prim(NC(M))$ as
well as the lengths of all nonzero words.
\end{proposition}

\noindent (Thus there are multiple primitive words for $d>2$.)
Using Proposition \ref{Pprim}, we address another difference with usual
nil-Coxeter algebras: the latter are always Frobenius \cite{Kho}. It is
natural to ask when the finite-dimensional algebras $NC_A(n,d)$ share
this property.

\begin{proposition}
The algebra $NC_A(n,d)$ is Frobenius if and only if $n=1$ or $d=2$.
\end{proposition}

In fact this happens if and only if the group algebra $\bk W(M_{A_n}({\bf
d}))$ is a flat deformation of $NC_A(n,d)$. Flat deformations will be
further discussed in the final section.

\begin{proof}
If $W(M)$ is a finite Coxeter group, \cite[\S 2.2]{Kho} shows that 
$NC(M)$ is Frobenius. Next, one easily verifies $NC_A(1,d) = \bk[T_1] /
(T_1^d)$ is Frobenius, via the symmetric bilinear form given by:
$\sigma(T_1^i, T_1^j) = {\bf 1}(i+j=d-1)$. Now suppose for some $n,d$
that $NC_A(n,d)$ is Frobenius, with nondegenerate invariant bilinear form
$\sigma$. For each primitive $p \neq 0$, there exists $a_p$ such that $0
\neq \sigma(p,a_p) = \sigma(p a_p, 1)$. Thus, we can take $a_p = 1, \
\forall p$. Since the functional $\sigma(-,1) : \Prim(NC_A(n,d)) \to \bk$
is nonsingular, we obtain $\dim_\bk \Prim(NC_A(n,d)) = 1$. Applying
Proposition \ref{Pprim}, we get $n=1$ or $d=2$.
\end{proof}

Finally, recall the famous result by Khovanov \cite{Kho} that the Weyl
algebra $W_n := \Z \langle x, \partial \rangle / (\partial x = 1 + x
\partial)$ can be represented by functors on bimodule categories over
usual nil-Coxeter algebras. (Here we use that the nil-Coxeter algebra
$\mathcal{A}_n := NC_A(n,2)$ is a bimodule over $\mathcal{A}_{n-1}$.)
We now explain how $NC_A(n,d)$ fits into Khovanov's framework for $d
\geqslant 2$, noting that for $d=2$ it was proved in \cite{Kho}:

\begin{proposition}[see \cite{Kh}]\label{Pkhovanov}
For $n \geqslant 1$ and $d \geqslant 2$, there is an isomorphism of
$\mathcal{A}_{n-1}$-bimodules:
\[
NC_A(n,d) \simeq \mathcal{A}_{n-1} \oplus \bigoplus_{k=1}^{d-1} \left(
\mathcal{A}_{n-1} \otimes_{\mathcal{A}_{n-2}} \mathcal{A}_{n-1} \right).
\]
\end{proposition}

For $d \geqslant 2$, in the notation of \cite{Kho} this result implies
that over $\mathcal{A}_{n-1}$-bimodules, the algebra
$NC_A(n,d)$ corresponds to $1 + (d-1) x \partial$. Thus, Proposition
\ref{Pkhovanov} strengthens Theorems \ref{ThmA} and \ref{ThmB}, which
discussed a left $\mathcal{A}_{n-1}$-module structure on $NC_A(n,d)$
(namely, $NC_A(n,d)$ is free of rank $1 + n(d-1)$).

\section{All finite-dimensional generalized nil-Coxeter algebras}

We conclude by proving Theorem \ref{ThmC}. Clearly $(2) \implies (1)$ by
Theorem \ref{ThmA} and \cite[Chapter 7]{Hum}. Now suppose $(1)$ holds and
${\bf d} \neq (2,\dots,2)$. We again use the diagrammatic calculus above,
now for the diagrams in Figure \ref{Fig3}.
\begin{figure}[ht]
\hspace*{14mm}\begin{tikzpicture}[line cap=round,line join=round,>=triangle 45,x=1.0cm,y=1.0cm]
\draw(5.8,13) circle (0.25cm);
\draw(7.5,14.1) circle (0.4cm);
\draw(9.5,14.1) circle (0.4cm);
\draw(11.5,14.1) circle (0.4cm);
\draw(13.5,14.1) circle (0.4cm);
\draw(13.5,11.9) circle (0.4cm);
\draw(11.5,11.9) circle (0.4cm);
\draw(9.5,11.9) circle (0.4cm);
\draw(7.5,11.9) circle (0.4cm);
\draw(14.9,13) circle (0.25cm);
\draw (5.5,13.3) node[anchor=north west] {$A$};
\draw (7.1,14.4) node[anchor=north west] {$B_1$};
\draw (9.1,14.4) node[anchor=north west] {$B_2$};
\draw (11,14.4) node[anchor=north west] {$B_{m'}$};
\draw (13.05,14.4) node[anchor=north west] {$B_m$};
\draw (13,12.2) node[anchor=north west] {$B'_m$};
\draw (11,12.2) node[anchor=north west] {$B'_{m'}$};
\draw (9.1,12.2) node[anchor=north west] {$B'_2$};
\draw (7.1,12.2) node[anchor=north west] {$B'_1$};
\draw (14.6,13.3) node[anchor=north west] {$C$};
\draw [->] (7,12.2) -- (6.1,12.8);
\draw [->] (6.1,13.2) -- (7,14);
\draw [->] (8,14.1) -- (9,14.1);
\draw [->] (12,14.1) -- (13,14.1);
\draw [->] (13,11.9) -- (12,11.9);
\draw [->] (9,11.9) -- (8,11.9);
\draw [->] (13.9,13.8) -- (14.6,13.2);
\draw [->] (14.6,12.8) -- (13.9,12.1);
\draw (6.2,12.5) node[anchor=north west] {$ \alpha $};
\draw (6.2,14) node[anchor=north west] {$ \alpha $};
\draw (8.1,14.7) node[anchor=north west] {$ \beta_1 $};
\draw (10.1,14.3) node[anchor=north west] {$ \cdots $};
\draw (12,14.7) node[anchor=north west] {$ \beta_{m'} $};
\draw (14.2,14.1) node[anchor=north west] {$ \gamma $};
\draw (14.2,12.5) node[anchor=north west] {$ \gamma $};
\draw (12.3,11.9) node[anchor=north west] {$ \beta_{m'} $};
\draw (10.1,12.1) node[anchor=north west] {$ \cdots $};
\draw (8.3,11.9) node[anchor=north west] {$ \beta_1 $};
\draw (6,13.3) node[anchor=north west] {+};
\draw (8.7,11.1) node[anchor=north west] {Fig. 3.1 ($m' = m-1$)};
\draw(3,5.7) circle (0.25cm);
\draw(5,5.7) circle (0.25cm);
\draw(3,3.7) circle (0.25cm);
\draw (2.7,5.95) node[anchor=north west] {\textit{A}};
\draw (4.7,5.95) node[anchor=north west] {\textit{B}};
\draw (2.7,3.95) node[anchor=north west] {\textit{C}};
\draw [->] (3,4.1) -- (3,5.3);
\draw [->] (3.4,5.7) -- (4.6,5.7);
\draw [->] (4.7,5.4) -- (3.3,4);
\draw (2.5,4.9) node[anchor=north west] {$ t $};
\draw (3.7,6.2) node[anchor=north west] {$ s $};
\draw (4,4.8) node[anchor=north west] {$ u $};
\draw (3.1,5.6) node[anchor=north west] {+};
\draw (3.4,3.6) node[anchor=north west] {Fig. 3.2}; 
\draw(5.8,7.7) circle (0.25cm);
\draw(7.5,8.8) circle (0.4cm);
\draw(9.5,8.8) circle (0.4cm);
\draw(11.5,8.8) circle (0.4cm);
\draw(13.5,8.8) circle (0.4cm);
\draw(13.5,6.6) circle (0.4cm);
\draw(11.5,6.6) circle (0.4cm);
\draw(9.5,6.6) circle (0.4cm);
\draw(7.5,6.6) circle (0.4cm);
\draw(12.1,7.7) circle (0.25cm);
\draw(14.9,7.7) circle (0.25cm);
\draw (5.5,8) node[anchor=north west] {$A$};
\draw (7.1,9.1) node[anchor=north west] {$B_1$};
\draw (9.1,9.1) node[anchor=north west] {$B_2$};
\draw (11,9.1) node[anchor=north west] {$B_{m'}$};
\draw (13.05,9.1) node[anchor=north west] {$B_m$};
\draw (13,6.9) node[anchor=north west] {$B'_m$};
\draw (11,6.9) node[anchor=north west] {$B'_{m'}$};
\draw (9.1,6.9) node[anchor=north west] {$B'_2$};
\draw (7.1,6.9) node[anchor=north west] {$B'_1$};
\draw (11.8,8) node[anchor=north west] {$D$};
\draw (14.6,8) node[anchor=north west] {$C$};
\draw [->] (7,6.9) -- (6.1,7.5);
\draw [->] (6.1,7.9) -- (7,8.7);
\draw [->] (8,8.8) -- (9,8.8);
\draw [->] (12,8.8) -- (13,8.8);
\draw [->] (13,6.6) -- (12,6.6);
\draw [->] (9,6.6) -- (8,6.6);
\draw [->] (13.1,8.6) -- (12.3,7.9);
\draw [->] (12.3,7.5) -- (13.1,6.9);
\draw [->] (13.9,8.5) -- (14.6,7.9);
\draw [->] (14.6,7.5) -- (13.9,6.8);
\draw (6.2,7.2) node[anchor=north west] {$ \alpha $};
\draw (6.2,8.7) node[anchor=north west] {$ \alpha $};
\draw (8.1,9.4) node[anchor=north west] {$ \beta_1 $};
\draw (10.1,9) node[anchor=north west] {$ \cdots $};
\draw (12,9.4) node[anchor=north west] {$ \beta_{m'} $};
\draw (12.7,8.4) node[anchor=north west] {$ \gamma $};
\draw (12.7,7.6) node[anchor=north west] {$ \delta $};
\draw (14.2,8.8) node[anchor=north west] {$ \delta $};
\draw (14.2,7.2) node[anchor=north west] {$ \gamma $};
\draw (12.3,6.6) node[anchor=north west] {$ \beta_{m'} $};
\draw (10.1,6.8) node[anchor=north west] {$ \cdots $};
\draw (8.3,6.6) node[anchor=north west] {$ \beta_1 $};
\draw (6,8) node[anchor=north west] {+};
\draw (8.7,5.8) node[anchor=north west] {Fig. 3.3 ($m' = m-1$)};
\draw(5.8,2.1) circle (0.25cm);
\draw(7.5,3.2) circle (0.4cm);
\draw(9.5,3.2) circle (0.4cm);
\draw(11.5,3.2) circle (0.4cm);
\draw(13.5,3.2) circle (0.4cm);
\draw(13.5,1) circle (0.4cm);
\draw(11.5,1) circle (0.4cm);
\draw(9.5,1) circle (0.4cm);
\draw(7.5,1) circle (0.4cm);
\draw (5.5,2.4) node[anchor=north west] {$A$};
\draw (7.1,3.5) node[anchor=north west] {$B_1$};
\draw (9.1,3.5) node[anchor=north west] {$B_2$};
\draw (11,3.5) node[anchor=north west] {$B_{m'}$};
\draw (13.05,3.5) node[anchor=north west] {$B_m$};
\draw (13,1.3) node[anchor=north west] {$B'_m$};
\draw (11,1.3) node[anchor=north west] {$B'_{m'}$};
\draw (9.1,1.3) node[anchor=north west] {$B'_2$};
\draw (7.1,1.3) node[anchor=north west] {$B'_1$};
\draw [->] (7,1.3) -- (6.1,1.9);
\draw [->] (6.1,2.3) -- (7,3.1);
\draw [->] (8,3.2) -- (9,3.2);
\draw [->] (12,3.2) -- (13,3.2);
\draw [->] (13.5,2.7) -- (13.5,1.5);
\draw [->] (13,1) -- (12,1);
\draw [->] (9,1) -- (8,1);
\draw (6.2,1.6) node[anchor=north west] {$ \alpha $};
\draw (6.2,3.1) node[anchor=north west] {$ \alpha $};
\draw (8.1,3.8) node[anchor=north west] {$ \beta_1 $};
\draw (10.1,3.4) node[anchor=north west] {$ \cdots $};
\draw (12,3.8) node[anchor=north west] {$ \beta_{m'} $};
\draw (13.6,2.4) node[anchor=north west] {$ \gamma $};
\draw (12.3,1) node[anchor=north west] {$ \beta_{m'} $};
\draw (10.1,1.2) node[anchor=north west] {$ \cdots $};
\draw (8.3,1) node[anchor=north west] {$ \beta_1 $};
\draw (6,2.4) node[anchor=north west] {+};
\draw (8.7,0) node[anchor=north west] {Fig. 3.4 ($m' = m-1$)};
\end{tikzpicture}
\caption{Modules for the infinite-dimensional generalized nil-Coxeter
algebras}
\label{Fig3}
\end{figure}

We consider the possible cases, showing in each case that the algebra
$NC(M)$ is infinite-dimensional, until we are left with only $NC_A(n,d)$.
First suppose there exist two nodes $\alpha, \gamma \in I$ with
$m_{\alpha \alpha}, m_{\gamma \gamma} \geqslant 3$. Since the Dynkin
diagram of $I$ is connected by assumption, there exist $\beta_1, \dots,
\beta_{m-1} \in I$ such that
\[
\alpha \quad \longleftrightarrow \quad \beta_1 \quad \longleftrightarrow
\quad \cdots \quad \longleftrightarrow \quad \beta_{m-1} \quad
\longleftrightarrow \quad \gamma
\]
are all connected in $I$, i.e., a path. Now define an $NC(M)$-module
$\scrm$ with basis
\[
A_r, B_{1r}, \dots, B_{mr}, C_r, B'_{1r}, \dots, B'_{mr},
\quad r \geqslant 1,
\]
and where every $T_i$ kills all basis vectors, \textit{except} for the
actions described in Figure 3.1, namely $T_\alpha(A_r) := B_{1r},
T_{\beta_1}(B'_{2r}) := B'_{1r}$, and so on for all $r \geqslant 1$. The
`$+$' indicates that $T_\alpha(B'_{1r}) := A_{r+1}\ \forall r \geqslant
1$. One verifies that the $T_i$ satisfy the $NC(M)$-relations on every
basis vector, whence on $\scrm$. Now as $\scrm$ is cyclic and
infinite-dimensional, so is $NC(M)$.

The strategy is similar for the remainder of the proof. Henceforth we fix
the unique node $\alpha \in I$ such that $m_{\alpha \alpha} \geqslant 3$.
If $\alpha$ is connected in $I$ to $\gamma$ with $m_{\alpha \gamma}
\geqslant 4$, then we work with Figure 3.2, setting $(s,t,u) \leadsto
(\alpha, \alpha, \gamma)$, and define $\scrm := {\rm span}_\bk \{ A_r,
B_r, C_r : r \geqslant 1 \}$. Now check that $\scrm$ is an
infinite-dimensional cyclic $NC(M)$-module. Next, if $\alpha$ is adjacent
to two nodes $\gamma, \delta \in I$, work with Figure 3.3 for $m=1$. This
shows $\alpha$ must be extremal.

Note that if $NC(M)$ is finite-dimensional then so is its quotient
$NC(M((2,\dots,2)))$, which is a nil-Coxeter algebra. Hence
$W(M((2,\dots,2)))$ is a finite Coxeter group, and these are known
\cite{Cox,Cox2}. We now sketch how to eliminate all cases not of type $A$,
whence from above, $NC(M) \cong NC_A(n,d)$, where we set $n := |I|$.

The dihedral types $G_2, H_2, I$ do not hold from above. Suppose $I$ is
of type $B,C,H$:
\[
\alpha \quad \longleftrightarrow \quad \beta_1 \quad \longleftrightarrow
\quad \cdots \quad \longleftrightarrow \quad \beta_{m-1} \quad
\longleftrightarrow \quad \gamma,
\]

\noindent with $m_{\alpha \alpha} \geqslant 3, m_{\gamma \gamma} = 2,
m_{\beta_{m-1} \gamma} \geqslant 4$. Then work with Figure 3.4. Note this
also rules out the $F_4$ case, since $NC(M_{F_4}) \twoheadrightarrow
NC(M_{B_3})$ or $NC(M_{C_3})$ by killing the extremal generator
$T_\delta$ with $m_{\delta \delta} = 2$, and we showed that the latter
two algebras are infinite-dimensional.

Next suppose $I$ is of type $D$. If $\alpha$ is a (extremal) node on the
`long arm', work with Figure 3.3 with $m=n-2$, with the other extremal
nodes $\gamma, \delta$.
Else if $\alpha$ is extremal on a short arm, we work as above with the
quotient algebra $NC(M_{D_4})$ of Dynkin type $D_4$, by killing all $T_j$
with node $j$ on the long arm having degree $\leqslant 2$. Working with
Figure 3.3 for $m=2$, it follows that $NC(M_{D_4})$, and hence $NC(M)$,
is infinite-dimensional.

Finally, in all remaining cases, the Coxeter graph of $I$ is of type $E$.
Akin to above, these cases are ruled out by quotienting to reduce to type
$D$. This concludes the proof.\qed





\begin{thebibliography}{88}
\bibitem{AnSc}
N.~Andruskiewitsch and H.-J.~Schneider.
\newblock On the classification of finite-dimensional pointed Hopf
algebras.
\newblock \href{http://dx.doi.org/10.4007/annals.2010.171.375}{\em Ann.
of Math.} 171(1):375--417, 2010.

\bibitem{BSS}
C.~Berg, F.~Saliola, and L.~Serrano, Pieri operators on the affine
  nilCoxeter algebra,
  \href{https://dx.doi.org/10.1090/S0002-9947-2013-05895-3}{\em Trans.
  Amer. Math. Soc.} 366(1):531--546, 2014.

\bibitem{BMR2}
M.~Brou\'e, G.~Malle, and R.~Rouquier.
\newblock Complex reflection groups, braid groups, Hecke algebras.
\newblock \href{https://dx.doi.org/10.1515/crll.1998.064}{\em J. reine
  angew. Math.} 500:127--190, 1998.

\bibitem{Cox}
H.S.M.~Coxeter.
\newblock Discrete groups generated by reflections.
\newblock \href{http://www.jstor.org/stable/1968753}{\em Ann. of Math.}
  35(3):588--621, 1934.

\bibitem{Cox2}
H.S.M.~Coxeter.
\newblock The complete enumeration of finite groups of the form
  $R_i^2 = (R_i R_j)^{k_{ij}} = 1$.
\newblock \href{http://dx.doi.org/10.1112/jlms/s1-10.37.21}{\em J. Lond.
  Math. Soc.} s1-10(1):21--25, 1935.

\bibitem{Co}
H.S.M.~Coxeter.
\newblock Factor groups of the braid group.
\newblock In: Proc. 4th Canadian Math. Congress (Banff, 1957), University
  of Toronto Press, 95--122, 1959.

\bibitem{Et}
P.~Etingof.
\newblock Proof of the Brou\'e--Malle--Rouquier conjecture in
characteristic zero (after I. Losev and I. Marin--G. Pfeiffer).
\newblock \href{http://dx.doi.org/10.1007/s40598-017-0069-7}{\em Arnold
  Math. J.} 3(3):445--449, 2017.

\bibitem{FS}
S.~Fomin and R.~Stanley.
\newblock Schubert polynomials and the nil-Coxeter algebra.
\newblock \href{http://dx.doi.org/10.1006/aima.1994.1009}{\em Adv. in
  Math.} 103(2):196--207, 1994.

\bibitem{HV2}
I.~Heckenberger and L.~Vendramin.
\newblock The classification of Nichols algebras over groups with finite
  root system of rank two.
\newblock \href{http://dx.doi.org/10.4171/JEMS/711}{\em J. Eur. Math.
  Soc.} 19(7):1977--2017, 2017.

\bibitem{Hum}
J.E.~Humphreys.
\newblock {\em Reflection groups and Coxeter groups}.
\newblock Cambridge studies in advanced mathematics no. \textbf{29},
  Cambridge University Press, Cambridge-New York, 1990.

\bibitem{Kh}
A.~Khare.
\newblock Generalized nil-Coxeter algebras over discrete complex
  reflection groups.
\newblock \href{http://dx.doi.org/10.1090/tran/7304}{\em Trans. Amer.
  Math. Soc.}, in press, 2017.

\bibitem{Kho}
M.~Khovanov.
\newblock NilCoxeter algebras categorify the Weyl algebra.
\newblock
  \href{http://www.tandfonline.com/doi/abs/10.1081/AGB-100106800}{\em
  Comm. Algebra} 29(11):5033--5052, 2001.

\bibitem{KK}
B.~Kostant and S.~Kumar.
\newblock The nil Hecke ring and cohomology of $G/P$ for a Kac-Moody
  group $G$.
\newblock \href{http://dx.doi.org/10.1016/0001-8708(86)90101-5}{\em Adv.
  in Math.} 62(3):187--237, 1986.

\bibitem{Ko}
D.W.~Koster.
\newblock {\em Complex reflection groups}.
\newblock Ph.D. thesis, University of Wisconsin, 1975.

\bibitem{Ma}
I.~Marin.
\newblock The freeness conjecture for Hecke algebras of complex
  reflection groups, and the case of the Hessian group $G_{26}$.
\newblock \href{http://dx.doi.org/10.1016/j.jpaa.2013.08.009}{\em J. Pure
  Appl. Algebra} 218(4):704--720, 2014.

\bibitem{Po1}
V.L.~Popov.
\newblock Discrete complex reflection groups.
\newblock \href{http://dx.doi.org/10.13140/2.1.4050.2088}{\em
  Communications of the Mathematical Institute}, vol.~15,
  Rijksuniversiteit Utrecht, 89 pp., 1982.

\bibitem{ST}
G.C.~Shephard and J.A.~Todd.
\newblock Finite unitary reflection groups.
\newblock \href{http://dx.doi.org/10.4153/CJM-1954-028-3}{\em Canadian
  J. Math.} 6:274--304, 1954.

\bibitem{SW2}
A.V.~Shepler and S.~Witherspoon.
\newblock A Poincar\'e--Birkhoff--Witt theorem for quadratic algebras
  with group actions.
\newblock
\href{http://www.ams.org/journals/tran/2014-366-12/S0002-9947-2014-06118-7/}{\em
  Trans. Amer. Math. Soc.} 366(12):6483--6506, 2014.
\end{thebibliography}
\end{document}